\documentclass[11pt]{amsart} 

\usepackage{amssymb}
\usepackage{mathtools}
\usepackage{mathrsfs}

\def\T{\mathbb T}
\def\upref#1{\textup{\ref{#1}}}
\def\q#1#2{#1/\negthickspace#2}
\newtheorem{lem}{Lemma}
\newtheorem{prp}{Proposition}
\newtheorem{thm}{Theorem}
\newtheorem{cor}{Corollary}

\usepackage{graphicx} 

\usepackage{array} 

\def\<{\langle}
\def\>{\rangle}

\def\Z{\mathbb Z}

\def\R{\mathbb R}

\def\Power{\mathcal P}

\title[Extensions of invariant orders]{Linear extensions of orders invariant under abelian group actions}
\author{Alexander R. Pruss}
\thanks{I am grateful to David Arnold, Dietrich Burde, Ramiro de la Vega, Trent Dougherty, A. Paul Pedersen and Friedrich Wehrung for discussions.}
\email{alexander\_pruss@baylor.edu}
\address{Baylor University}

\begin{document}
\begin{abstract}
Let $G$ be an abelian group acting on a set $X$, and suppose that no element of $G$ has any finite orbit of size greater than one. We show that every partial order on $X$ invariant under $G$ extends to a linear order on $X$ also invariant under $G$. We then discuss extensions to linear preorders when the orbit condition is not met, and show that for any abelian group acting on a set $X$, there is a linear preorder $\le$ on the powerset $\Power X$ invariant under $G$ and such that if $A$ is a proper subset of $B$, then $A<B$ (i.e., $A\le B$ but not $B\le A$).
\end{abstract}

\maketitle
\section{Linear orders}
Szpilrajn's Theorem~\cite{Szpilrajn} (proved independently by a number of others) says that given the Axiom of Choice, any partial order can be extended to a linear order, where $\le^*$ extends $\le$ provided that $x\le y$ implies $x\le^* y$.  There has been much work on what properties of the partial order can be preserved in the linear order (e.g., \cite{BonnetPouzet,DHLS,Yang}) but the preservation of symmetry under a group acting on partially ordered set appears to have been neglected.

Suppose a group $G$ acts on a partially ordered set $(X,\le)$ and the order is $G$-invariant, where a relation $R$ is $G$-invariant provided that for all $g\in G$ and $x,y\in X$, we have $xRy$ if and only if $(gx)R(gy)$. It is natural to ask about the condition under which $\le$ extends to a $G$-invariant linear order.  We shall answer this question in the case where $G$ is abelian. Then we will discuss extensions where the condition is not met. In the latter case, the extension will be to a linear preorder (total, reflexive and transitive relation) but will nonetheless preserve strict comparisons. Finally, we will apply the results to show that for any abelian group $G$ acting on a set $X$, there is a $G$-invariant linear preorder on the powerset $\Power X$ preserving strict set inclusion.

Throughout the paper we will assume the Axiom of Choice and all our proofs will be elementary and self-contained.

An \textit{orbit} of $g\in G$ is any set of the form $\{ g^n x : n \in \Z \}$.
An obvious necessary condition for $X$ to have a $G$-invariant linear order is that no element of $G$ has
any finite orbit of size greater than $1$.  Surprisingly, this is sufficient not just for the existence of an invariant linear order, but for invariant partial orders to have invariant linear extensions.

\begin{thm}\label{th:ext0} Let $G$ be an abelian group. The following are equivalent:
\begin{enumerate}
\item[\textup{(i)}] No element of $G$ has any finite orbit of length greater than one
\item[\textup{(ii)}] There is a $G$-invariant linear order on $X$
\item[\textup{(iii)}] Every $G$-invariant partial order on $X$ extends to a $G$-invariant linear order.
\end{enumerate}
\end{thm}

We will call (iii) the \textit{invariant order extension property}.

Theorem~\ref{th:ext0} yields a positive answer to de la Vega's question~\cite{delaVega} whether given an order automorphism $f$ of a partially ordered set $(X,\le)$, with $f$ having no finite orbits, $\le$ can be extended to a linear order $\le^*$ in such a way that $f$ is an order automorphism of $(X,\le^*)$. Just let $G$ be the group generated by $f$.

Both of the non-trivial implications in Theorem~\ref{th:ext0} are false for non-abelian groups. Any torsion-free group that is non-right-orderable~\cite{DPT,Promislow} acting on itself would provide a counterexample to (i)$\Rightarrow$(ii) while the fundamental group of the Klein bottle acting on itself would be a counterexample to (ii)$\Rightarrow$(iii)~\cite{DDHPV}.

For the proof of the theorem, define a relation  $\sim_G$ (or $\sim_{G,X}$ if we need to make $X$ clear) on $X$ by $x \sim_G y$ if and only if there is a $g\in G$ such that $g^n y = y$ for some $n\in\Z^+$ and $gy = x$.  Clearly $\sim_G$ is reflexive. To see that it is symmetric observe that if $g^n y = y$ and $gy = x$, then
$$
g^n x = g^{n+1} g^{-1} x = g^{n+1} y = g y = x,
$$
so $g^{-n} x = x$ and $x = g^{-1} y$.  If $G$ is abelian, $\sim_G$ is transitive. For if $g^m y = y$ and $gy = x$, and $h^n z = z$ and $hz = y$, then $(gh)z = x$ and
$$
(gh)^{mn+1} z = g g^{mn} h h^{mn}z=g g^{mn} h z = g g^{mn} y = gy = x.
$$

Also, given a $G$-invariant partial order $\le$, define the relation $\le_G$ by $x \le_G y$ if and only if there is a finite sequence $(g_i)_{i=1}^n$ in $G$ such that $x \le g_i y$ and $\prod_{i=1}^n g_i = e$.

Since in Theorem~\ref{th:ext0}, (iii)$\Rightarrow$(ii)$\Rightarrow$(i) is trivial, the theorem follows immediately from applying the following to a maximal $G$-invariant partial order on $X$ extending $\le$, which exists by Zorn, and obtaining a contradiction if that order is not linear.

\begin{prp}\label{prp:precise} Let $G$ be an abelian group acting freely on $X$. Let $\le$ be a $G$-invariant partial order.
If $\le$ is not a linear order and $G$ has no orbits of finite size greater than one, there exist $x$ and $y$ with $y\not\le_G x$ and $x\not\le y$.  Moreover, whenever $x$ and $y$ in $G$ are such that $y\not \le_G x$, then there is a $G$-invariant partial order $\le^*$ extending $\le$ such that $x \le^* y$.
\end{prp}

We now need to prove Proposition~\ref{prp:precise}.
Recall that $R$ is \textit{antisymmetric} provided that $xRy$ and $yRx$ implies $x=y$, so a partial order is an antisymmetric preorder.
We then need:

\begin{lem}\label{lem:prec}  Suppose $G$ is abelian and $\le$ is a $G$-invariant partial order. Then:
\begin{enumerate}
\item[\textup{(i)}] $\le_G$ is a $G$-invariant preorder extending $\le$.
\item[\textup{(ii)}] For all $x,y\in X$, the following are equivalent:
\begin{enumerate}
\item[\textup{(a)}] $x\sim_G y$
\item[\textup{(b)}] there is a finite sequence $(g_i)_{i=1}^n$ in $G$ such that $x = g_i y$, $1\le i\le n$, and $\prod_{i=1}^n g_i = e$
\item[\textup{(c)}] $x \le_G y$ and $y \le_G x$.
\end{enumerate}
\item[\textup{(iii)}] The following are equivalent:
\begin{enumerate}
\item[\textup{(a)}] $\le_G$ is antisymmetric
\item[\textup{(b)}] for all $x,y\in X$, $x \sim_G y$ implies $x=y$
\item[\textup{(c)}] no element of $G$ has any finite orbit of size greater than one.
\end{enumerate}
\end{enumerate}
\end{lem}

\begin{proof}[Proof of Lemma~\ref{lem:prec}] (i): Invariance and  reflexivity are clear.
Suppose $x \le g_i y$, $1\le i\le m$, and $y \le h_j z$, $1\le j\le n$, with the product of the $g_i$ being $e$ and that of the
$h_j$ being $e$ as well.  Then $g_i y \le g_i h_j z$ by $G$-invariance of $\le$,
so $x \le g_i h_j z$,
and it is easy to see that the product of all
the $g_i h_j$ is $e$, so $x \le_G y$. Finally, if $x \le y$, then $x \le ey$ and so $x \le_G y$.

(ii)(a)$\Rightarrow$(b): Assume (a). Then $g^n y = y$ and $x=gy$ for some $n\in\Z^+$ and $g\in G$, so
$g^{-n} y = y$ and $x = g^{1-n} y$.  Let $g_1 = g^{1-n}$ and let $g_i = g$ for $2\le i\le n$.
Then $x = g_i y$ for all $i$ and the product of the $g_i$ is $e$.

(ii)(b)$\Rightarrow$(a): Suppose $(g_i)_{i=1}^n$ in $G$ are such that $x = g_i y$ and $\prod_{i=1}^n g_i = e$.
Let $G_y$ be the stabilizer of $y$, i.e., the subgroup $\{ g \in G : gy = y \}$.  We have $g_i^{-1} g_j y = g_i^{-1} x = y$ for all $i,j$, so the cosets $[g_i]=g_i G_y$ and $[g_j] = g_j G_y$ in $G/G_y$ are equal for all $i,j$.  Thus, $[g_1^n] = [\prod_{i=1}^n g_i] = e$, and so $g_1^n \in G_y$.  Thus, $g_1^n y = y$ and $g_1 y = x$, so $x\sim_G y$. (I am grateful to Friedrich Wehrung for drawing my attention to the stabilizer subgroups in connection with condition (ii)(b).)

(ii)(b)$\Rightarrow$(c): Suppose $x = g_i y$ where the product of the $g_i$ is $e$.
Thus $x\le g_i y$ for all $i$, and $x\le_G y$. Let $h_i=g_i^{-1}$.  Then $y = h_i x$, so $y\le h_i x$, and the product of the $h_i$ is $e$, so $y\le_G x$.

(ii)(c)$\Rightarrow$(b): Suppose $x \le_G y$ and $y\le_G x$.  Suppose thus $x \le g_i y$, $1\le i\le m$, and $y \le h_i x$, $1\le i\le n$, with $\prod_{i=1}^m g_i =\prod_{i=1}^n h_i = e$.  By invariance, we have $g_i y \le g_i h_j x$, for $i\le m$ and $j\le n$, so
\begin{equation}\label{eq:ggh}
x \le g_i y \le g_i h_j x.
\end{equation}
Fix $1\le i_1 \le m$.
Let $(i_k,i_k), 1\le k\le mn$, enumerate $([1,m]\cap\Z) \times ([1,n]\cap\Z)$.  Then by iterating
\eqref{eq:ggh} and using the invariance of $\le$:
$$    x \le g_{i_1} y \le g_{i_1} h_{i_1} x \le g_{i_1} h_{i_1} g_{i_2} h_{i_2} x \le \dotsb \\
    \le \prod_{k=1}^{mn} (g_{i_k} h_{i_k}) x = x.
$$
Thus, $x = g_{i_1} y$. But $i_1$ was arbitrary. Thus, $x = g_i y$ for all $i$, and so $x\sim_G y$.

(iii): The equivalence of (a) and (b) follows from (ii). An element $g$ has an orbit of finite size greater than $1$ if and only if there is an $x$ such that $gx \ne x$ but $g^n x = x$ for some $n$. The equivalence of (b) and (c) follows.
\end{proof}

\begin{proof}[Proof of Proposition~\upref{prp:precise}]
If $\le$ is not a linear order, there are $x$ and $y$ such that $x\not\le y$ and $y\not\le x$.  By the antisymmetry of $\le_G$ (from Lemma~\ref{lem:prec}), at least one of $y\not\le_G x$ or $x\not\le_G y$ must also hold.

Suppose now that $y\not\le_G x$.

Let $a\le^0 b$ providing that either $a\le b$ or there is a $g\in G$ such that $a=gx$ and $b=gy$.

Let $\le^*$ be the transitive closure of $\le^0$. Then $\le^*$ is $G$-invariant, reflexive, transitive and an extension of $\le$. We need only show $\le^*$ to be antisymmetric.

Since $\le^*$ is the transitive closure of $\le^0$ while $\le$ is antisymmetric and transitive, if $\le^*$ fails to be antisymmetric, by definition of $\le^0$, there will have to be a loop of the form:
$$
   g_1 x \le^0 g_1 y \le g_2 x \le^0 g_2 y \le \dotsb \le g_n x \le^0 g_n y \le g_1 x.
$$
Let $g_{n+1}=g_1$. Thus, $g_i y \le g_{i+1} x$ for $1\le i\le n$.
By $G$-invariance, $y \le g_{i}^{-1} g_{i+1} x$.
Let $h_i = g_{i}^{-1} g_{i+1}$, so $y \le h_i x$, and observe that $\prod_{i=1}^n h_i = e$. Thus,
$y\le_G x$ by Lemma~\ref{lem:prec}, contrary to what we have assumed.
\end{proof}

Proposition~\ref{prp:precise} also yields:

\begin{cor}\label{cor:isect}
If $G$ is an abelian group acting on a set $X$ with a $G$-invariant partial order $\le$, and no element of $G$ has a finite orbit of size greater than one, then $\le_G$ is the intersection of all $G$-invariant linear orders extending $\le$.
\end{cor}

\begin{proof} Proposition~\ref{prp:precise} and Zorn's lemma shows that if $y\not\le_G x$, then there is a $G$-invariant linear order $\le^*$ extending $\le$ and such that $x\le^* y$ and hence such that $y\not\le^* x$.
Thus the intersection of all $G$-invariant linear orders extending $\le$ is contained in $\le_G$.

For the other inclusion, we need to show that if $\le^*$ is a $G$-invariant linear order extending $\le$,
then $x\le_G y$ implies $x\le^* y$.

Suppose $x\le_G y$, so there are $(g_i)_{i=1}^n$ whose product is $e$ and which satisfy $x \le g_i y$.
To obtain a contradiction, suppose $x\not\le^* y$. Since $\le^*$ is linear, $x\ne y$ and $y\le^* x$.
Thus, $x \le g_i y \le^* g_i x$ for all $i$.  Hence, using the invariance of $\le^*$ and iteratively
applying $x\le^* g_i x$:
$$
   x \le g_1 y \le^* g_1 x \le^* g_1 g_2 x \le^* \cdots \le^* g_1 g_2 \cdots g_n x = x.
$$
Thus $x=g_1 y$. Reordering the $g_i$ as needed, we can prove that $x=g_i y$ for all $i$, and so $x\sim_G y$ and
hence $x=y$ by Lemma~\ref{lem:prec}, contrary to our assumptions.
\end{proof}

Note that if $G$ is a partially ordered torsion-free abelian group considered as acting on itself, then it is easy to see that $x \le_G y$ if and only if there is an $n\in\Z^+$ such that $x^n \le_G y^n$.  Thus, if $\le$ is a \textit{normal} order in the terminology of \cite{Fuchs}, i.e., one such that $0 \le y^n$ implies $0\le y$ (and hence $x^n\le y^n$ implies $x\le y$), then $\le_G$ coincides with $\le$, and Corollary~\ref{cor:isect} yields classic results~\cite{Everett,Fuchs} on extensions of partial orders on abelian groups.

\section{Preorders and orderings of subsets}
Even if $G$'s action on $X$ lacks the invariant order extension property, we can extend a partial order to a linear preorder (i.e., a preorder where all elements are comparable).  Of course this is trivially true: just take the preorder such that for all $x,y$ we have $x\le^* y$ and $y\le^* x$.  What is not trivially true is that if $G$ is any abelian group, we can extend the partial order to a preorder while preserving all the strict inequalities in the partial order.  In fact, this is even true if we start off with $\le$ a preorder.  Recall that $x < y$ is defined to hold if and only if $x \le y$ and not $y \le x$.

\begin{thm}\label{thm:preorder}
If $G$ is any abelian group acting on a space $X$, and $\le$ is a $G$-invariant preorder on $X$, then
there a $G$-invariant linear preorder $\le^*$ on $X$ that extends $\le$ and is such that if $x < y$,
then $x <^* y$.
\end{thm}

The proof depends on two lemmas.

\begin{lem}\label{lem:quotient}
Suppose $G$ is an abelian group acting on a space $X$.  Let $Y=\q{X}{\sim_{G,X}}$ and extend the action of $g$
to $Y$ by $g[A]=[gA]$.  This is a well-defined group action and $G$ acting on $Y$ has the invariant order extension property.
\end{lem}

\begin{proof}
That the group action is well-defined follows from the fact that $x \sim_{G,X} y$ if and only if $gx \sim_{G,X} gy$, for any $x,y\in X$ and $g\in G$.

Suppose that $[x] \sim_{G,Y} [y]$ for $x,y\in X$.
Choose $f\in G$ and $m\in \Z+$ such that $f[y] = [x]$ and $f^m [y] = [y]$.
Without loss of generality assume $m\ge 3$.
Thus, $x \sim_{G,X} fy$ and $y \sim_{G,X} f^m y$.
Hence there are $g,h\in G$ and $n,p\in\Z^+$ such that
$gfy = x$, $g^n fy = fy$, $hf^m y = y$ and $h^p f^m y = f^m y$. Without loss of generality assume
$n\ge 3$.

  Thus, $y= g^{-n} y$, $y = f^{-m}h^{-1} y$ and $y = h^p y$.  Since $x=fgy$, we have $x = h_i y$, for $1\le i\le 4$, where:
\begin{align*}
h_1 &= fg\\
h_2 &= f g^{1-n}\\
h_3 &= f^{1-m} g h^{-1}\\
h_4 &= f g h^p.
\end{align*}
Let $n_1 = m(n-1)p - n(p+1)$, $n_2 = mp$, $n_3 = np$ and $n_4 = n$ (the values were generated by computer).  Given that $m\ge 3$ and $n\ge 3$, we have $n_1\ge 0$. Straightforwardly we have $h_1^{n_1} h_2^{n_2} h_3^{n_3} h_4^{n_4} = e$.  Then let the $g_i$ be a sequence of $n_1+n_2+n_3+n_4$ entries from $G$, with the first $n_1$ being all equal to $h_1$, the next $n_2$ being $h_2$, the next $n_3$ being $h_3$ and the rest being $h_4$.
Then $x = g_i y$ and the product of the $g_i$ is $e$, so $x\sim_{G,X} y$.
Thus $[x]=[y]$ and so we have the invariant order extension property.
\end{proof}

\begin{lem}\label{lem:ord}
Suppose $G$ is an abelian group acting on a space $X$ and $\le$ is a $G$-invariant partial order on $G$.
\begin{enumerate}
\item[\textup{(i)}] If $x < y$, then we do not have $x\sim_G y$
\item[\textup{(ii)}] If $x\sim_G x'$, $y\sim_G y'$ and $x\le y$, then $x' \le_G y'$
\item[\textup{(iii)}] If $x < y$ and $x\sim_G x'$ and $y\sim_G y'$, then we do not have $y' \le x'$.
\end{enumerate}
\end{lem}

\begin{proof}[Proof of Lemma~\upref{lem:ord}]
(i): Suppose $x < y$.  To obtain a contradiction, suppose $x\sim_{G} y$, so $gy = x$ and $g^n y = y$ for
some $n$ and $g$. By invariance, $g^k y > g^k x$ for all $k$. Thus:
$$
   y > x = gy >  gx = g^2 y > \dotsb > g^{n-1} x = g^n y = y,
$$
a contradiction.

(ii): If $x\sim_G x'$ and $y\sim_G y'$, then by Lemma~\ref{lem:prec} there are $(g_i)_{i=1}^m$, with product $e$, and $(h_i)_{i=1}^n$, with product $e$, such that $x = g_i x'$ and $y=h_j y'$.
Thus, $g_i x' \le h_j y'$ and by $G$-invariance of $\le$, we have $x' \le g_i^{-1} h_j y'$.  The product of the $g_i^{-1} h_j$, as $(i,j)$ ranges over $([1,m]\cap \Z)\times ([1,n]\cap \Z)$, is $e$, so $x' \le_G y'$.

(iii): Now suppose that $x < y$, $x\sim_G x'$ and $y\sim_G y'$.
Then $x' \le_G y'$ by (ii).  To obtain a contradiction, suppose $y'\le x'$.  So $y'\le_G x'$. Thus, $x'\sim_G y'$ by Lemma~\ref{lem:prec}. Since $\sim_G$ is an equivalence relation,
$x\sim_G y$, which contradicts $x<y$ by (i).
\end{proof}

\begin{proof}[Proof of Theorem~\upref{thm:preorder}]
First note that we only need to prove the result for $\le$ a partial order.  For if $\le$ is a preorder, then we can replace $X$ by $\q{X}{\simeq}$ where $x\simeq y$ if and only if $x\le y$ and $y\le x$.  Define the natural group action of $G$ by $g[x]_\simeq = [gx]_\simeq$ and note that stipulating that $[x]_\simeq\preceq [y]_\simeq$ if and only if $x\le y$ gives a well-defined $G$-invariant partial order.  The partial order version of the theorem then yields a linear preorder extending $\preceq$, which lifts to a linear preorder on $X$ satisfying the required conditions.

Suppose thus that $\le$ is a $G$-invariant partial order on $X$.  For $a,b\in Y=\q{X}{\sim_{G,X}}$, let $a\le^0 b$ if and only if there are representatives $x\in a$ and $y\in b$ such that $x \le y$.

Clearly, $\le^0$ is reflexive and $G$-invariant.  Suppose that $a\le^0 b$ and $b\le^0 c$. Choose $x\in a$, $y_1,y_2\in b$ and $z\in c$ such that $x\le y_1$ and $y_1\le z$.  Since $y_1\sim_G y_2$, by Lemma~\ref{lem:prec} we have $y_1\le y_2$, so $x\le z$ and $a\le^0 c$.

We now check that $\le^0$ is antisymmetric.  Suppose $a \le^0 b$ and $b\le^0 a$.
Thus there are representatives $x,x'\in a$ and $y,y'\in b$ such
that $x \le y$ and $y'\le x'$.  If $x=y$, we have $a=b$ as desired.
Otherwise, $x<y$.
Moreover, $x\sim_{G,X} x'$ and $y\sim_{G,X} y'$. But that would contradict Lemma~\ref{lem:ord}(iii).

Thus $\le^0$ is a partial order.
By Lemma~\ref{lem:quotient} and Theorem~\ref{th:ext0},
extend it to a $G$-invariant linear order $\le^1$ on $Y$.  Now let $x \le^* y$ if and only if $[x] \le^1 [y]$.  This is a $G$-invariant linear preorder.

Suppose $x < y$.  We then have $[x]\le^1 [y]$.  Thus $x \le^* y$.  To complete our proof, we must show
$y\not\le^* x$.  By Lemma~\ref{lem:ord}(i), we do not have $x \sim_{G,X} y$, and so $[x]\ne [y]$.
Since $\le^1$ is a partial order, $[y]\not\le^1 [x]$ and so $y\not\le^* x$.  Thus $x <^* y$.
\end{proof}

\begin{cor}\label{cor:comp} Suppose $G$ is an abelian group acting on a space $X$. Then there is a $G$-invariant linear preorder $\le$ on the powerset $\Power X$ such that if $A$ is a proper subset of $B$, then $A < B$.
\end{cor}

In particular, there is a translation-invariant ``size comparison'' for subsets of $\R^n$ for all $n$ as well as a rotationally-invariant ``size comparison'' for subsets of the circle $\T$ that preserves the intuition that proper subsets are ``smaller''.

Corollary~\ref{cor:comp} is not true in general for non-abelian $G$, even in the case of isometry groups that are ``very close'' to abelian.  For instance, suppose $G$ is all isometries on the line $\R$.  This has the translations as a subgroup of index two and is supramenable, i.e., for every non-empty subset $A$ of any set $X$ that it acts on, there is a finitely-additive $G$-invariant measure $\mu$ of $X$ with $\mu(A)=1$~\cite[Chapter~12]{Wagon}.  But we shall shortly see that there is no $G$-invariant preorder $\le$ on $\Power\R$ such that $A<B$ whenever $A$ is a proper subset of $B$.

To see this, say that a preorder $\le$ is strongly $G$-invariant provided that $x\le y$ if and only if $gx \le y$ if and only if $x\le gy$ for all $g\in G$ and $x,y\in X$.
Then there is no \textit{strongly} $G$-invariant preorder $\le$ on $\Power\R$ such that $A\subset B$ implies $A < B$, since if $\le$ were such a preorder, then we would have $\Z^+ < \Z_0^+$ even though $1+\Z^+ = \Z_0^+$.

But it turns out that if $G$ is all isometries on $\R$, then invariance implies strong invariance, and so there is no invariant $G$-invariant preorder on $\Power\R$ which preserves strict inclusion.  For the isometry group $G$ is generated by elements of finite order, namely reflections, and elements of finite order have finite orbits, while:

\begin{prp}\label{prp:gen} If $\le$ is a $G$-invariant linear preorder on $X$, and $G$ is any group generated by elements all of whose
orbits are finite, then $\le$ is strongly $G$-invariant.
\end{prp}

\begin{proof} We only need to prove that if $g\in G$ has only finite orbits, then $x\le y$ implies $gx\le y$.
Suppose $x\le y$ and $g^n x = x$.
By linearity, we have $x \le gx$ or $gx \le x$ (or both).  If $x \le gx$, then $g^k x \le g^{k+1} x$ for all $k$
by invariance, and so
$$
    x \le gx \le g^2 x \le \dotsb \le g^n x = x
$$    
and so $gx\le x$.  So in either case, $gx\le x$. By transitivity, $x\le y$ implies $gx\le y$.
\end{proof}

The following generalizes the remarks about the isometries on $\R$:

\begin{cor} If $G$ is any group acting on a set $X$ and there are $g,h\in G$ with only finite orbits, while
$gh$ has at least one infinite orbit, then there is no $G$-invariant preorder $\le$ on $\Power X$ such that if $A$ is a proper subset of $B$, then $A<B$.
\end{cor}

\begin{proof} Without loss of generality, $G$ is generated by $g$ and $h$. Let $A$ be an infinite orbit of
$gh$, fix $x\in A$ and let $A^+ = \{ (gh)^n x : n \in \Z_0^+ \}$. Then $ghA^+$ is a proper subset of $A^+$, and there is no strongly $G$-invariant preorder $\le$ on $\Power X$ such that $ghA^+<A^+$. By Proposition~\ref{prp:gen}, there is no $G$-invariant preorder like that, either.
\end{proof}

\end{document}